\documentclass[11pt]{amsart}  
\usepackage{amsmath,amsfonts,amssymb,mathrsfs}
\usepackage[english]{babel}
\usepackage[utf8]{inputenc}
\usepackage[T1]{fontenc}
\usepackage{pgf,tikz,pgfplots}
\pgfplotsset{compat=1.15}
\usepackage{mathrsfs}
\usepackage{bbold}
\usepackage{enumitem}
\usetikzlibrary{arrows}
\newtheorem{theorem}{Theorem}[section]
\newtheorem{thm}[theorem]{Theorem}
\newtheorem{lemma}[theorem]{Lemma}

\newtheorem{remark}{Remark}[section]

\def\R{\mathbb R}

\def\N{\mathbb N}
\def\Z{\mathbb Z}

\numberwithin{equation}{section}

\usepackage{hyperref}

\DeclareMathOperator{\dive}{div}

\usepackage{tikz}
\usetikzlibrary{decorations}
\usetikzlibrary{decorations.pathmorphing}
\usetikzlibrary{decorations.pathreplacing}
\usetikzlibrary{decorations.shapes}
\usetikzlibrary{decorations.text}
\usetikzlibrary{decorations.markings}
\usetikzlibrary{decorations.fractals}
\usetikzlibrary{decorations.footprints}
\tikzset{every picture/.style={execute at begin picture={
   \shorthandoff{:;!?};}
}}
\begin{document}
\title[Primitive equations for the ocean and atmosphere]{From anisotropic Navier-Stokes equations to primitive equations for the ocean and atmosphere}

\author[Valentin Lemarié]{Valentin Lemarié}
 {\begin{center}
\begin{abstract} 
We study the well-posedness of the primitive equations for the ocean and atmosphere on two particular domains : a bounded domain $\Omega_1\mathrel{\mathop:}=(-1,1)^3$ with periodic boundary conditions and the strip $\Omega_2\mathrel{\mathop:}=\R^2\times(-1,1)$ with a periodic boundary condition for the vertical coordinate. An existence theorem for global solutions on a suitable Besov space is derived. 
Then, in a second step, we rigorously justify the passage to the limit from the rescaled anisotropic Navier-Stokes equations to these primitive equations in the same functional framework as that found for the solutions of the primitive equations. 
 \end{abstract}\end{center}}
\maketitle

\section{Introduction}
The primitive equations for the large-scale dynamics of the ocean and atmosphere were introduced in 1922 by L.F.Richardson \cite{Richardson} : the latter play a fundamental role in geophysical fluid dynamics \cite{Haltiner}, \cite{Lewandowski}, \cite{Majda}, \cite{Pedlosky}, \cite{Vallis}, \cite{Washington} and \cite{Zeng}. They were then applied to atmospheric models by Smagorinsky \cite{Smagorinsky} and oceanography by  Bryan \cite{Bryan}.
We refer to the various sources cited for the physical aspect of the system.

In this article, we will mathematically study these primitive equations for the ocean and atmosphere on $$\Omega_1\mathrel{\mathop:}=(-1,1)^3, \quad \text{or} \quad \Omega_2\mathrel{\mathop:}=\R^2\times(-1,1) :$$ \begin{equation}\label{Equations primitives}\left\{\begin{array}{l}
   \partial_t v + u\cdot \nabla v - \Delta v +\nabla_H p=0,  
   \\
    \partial_z p =0,
    \\
    \dive_H v+\partial_z w =0,
    \\ v \ \text{even (resp $w$ odd) w.r.t the vertical coordinate} \ z,
\end{array}\right.\end{equation}
where $u=(v,w)$ is periodic for $\Omega_1$ (resp. periodic w.r.t the vertical coordinate $z$ for $\Omega_2$) with $v$ the horizontal component and $w$ the vertical component, $\nabla_H\mathrel{\mathop:}= \begin{pmatrix}
    \partial_1 \\ \partial_2
\end{pmatrix}$ the horizontal gradient and $\dive_H V\mathrel{\mathop:}=\partial_1 V_1+\partial_2 V_2$ the horizontal divergence. 

We will refer to $\Omega$ the space domain (referring to $\Omega_1$ or to $\Omega_2$) and $\Omega_h$ (referring to $(-1,1)^2$ or  $\R^2$).

The mathematical analysis of these equations dates back to the work of J.-L. Lions, Temam and Wang \cite{Lions1}, \cite{Lions2}, \cite{Lions3} in the 1990s, who studied the existence of global weak solutions (without uniqueness) for these equations coupled to the temperature equation on a spherical envelope. Other results have been proved for the primitive equations by adding a Coriolis force: for initial data in $H^1$, Guillén-Gonzalez, Masmoudi and Rodriguez-Bellido \cite{Guillen} proved the local well-posedness of the problem and later with an energy bound $H^1$, Cao and Titi \cite{Titi} obtained the globally well-posed character of strong solutions in dimension 3 in a more general framework where temperature is considered. 

More recently, results of global solutions in spaces of type $L^2$ (based on maximum regularity techniques) have been obtained by Hieber et al. \cite{Hieber1}, \cite{Hieber2} and Giga et.al \cite{Giga1}, \cite{Giga2} who consider the system \eqref{Equations primitives}.

All these results have been proved on a bounded domain with periodic boundary conditions, a lot of regularity and the solutions are only local in time. We propose here a study for an initial data in the Besov space $\dot B_{2,1}^{\frac{1}{2}}\cap\dot B_{2,1}^{\frac{3}{2}}$. We prove the existence and uniqueness of global solutions on the $\Omega$ domain, possibly unbounded horizontally, where we impose conditions on the vertical component (a periodic condition on this direction and a parity condition on the vertical component of the solution). 

Secondly, we want to rigorously justify the hydrostatic approximation : the system \eqref{Equations primitives} can be formally obtained from the Navier-Stokes equations as follows. 
Let us consider the anisotropic Navier-Stokes equations on the thin domain $\Omega_{1,\varepsilon}=(-1,1)^2\times(-\varepsilon,\varepsilon)$ or $\Omega_{2,\varepsilon}=\R^2\times (-\varepsilon,\varepsilon)$  : 
\begin{eqnarray}\label{NS}\left\{\begin{array}{l}
    \partial_t \tilde{u}+\tilde{u}\cdot \nabla \tilde{u}-\mu_H  \Delta_H \tilde{u}-\mu_z \partial_z^2 \tilde{u}+\nabla \tilde{p}=0 \\ \dive \tilde{u}=0
\end{array} \right.\end{eqnarray} with $\mu_H=1$ and $\mu_z=\varepsilon^2$. Introducing new unknowns $$\displaylines{v_\varepsilon(x,y,z,t)\mathrel{\mathop:}=(\tilde{u}_1,\tilde{u}_2)(x,y,\varepsilon z,t), \ w_\varepsilon(x,y,z,t)\mathrel{\mathop:}=\varepsilon^{-1}\tilde{u}_3(x,y,\varepsilon z,t), \cr
 u_\varepsilon\mathrel{\mathop:}=(v_\varepsilon,w_\varepsilon), \ p_\varepsilon(x,y,z,t)\mathrel{\mathop:}=\tilde{p}(x,y,\varepsilon z,t),}$$ we can rewrite \eqref{NS} like
\begin{equation}\label{NS remise à l'échelle}
    \left\{\begin{array}{l}
     \partial_t v_\varepsilon+u_\varepsilon\cdot \nabla v_\varepsilon-\Delta v_\varepsilon+\nabla_H p_\varepsilon=0  \\
     \varepsilon^2\left(\partial_t w_\varepsilon+u_\varepsilon\cdot \nabla w_\varepsilon-\Delta w_\varepsilon\right)+\partial_z p_\varepsilon=0 \\
     \dive u_\varepsilon=0 \\
     v_\varepsilon \ \text{even (resp $w_\varepsilon$ odd) w.r.t the vertical coordinate} \ z,
\end{array}\right. \end{equation} on the domain $\Omega$ independent of $\varepsilon$ with the same periodicity condition on $u_\varepsilon$ as system \eqref{Equations primitives}.

Formally, taking the limit when $\varepsilon$ tends to 0 in \eqref{NS remise à l'échelle}, we obtain the primitive equations \eqref{Equations primitives}.

On the 3-dimensional torus, this passage to the limit has been justified locally in time by Hieber et al. in \cite{article de référence} with techniques using maximum parabolic regularity. We obtain here a justification on the same space as the study of primitive equations, globally in time and for less regular data. 

\section{Main results and strategy of proof}
In this section, we first explain notations and definitions used in this article, describe the results obtained and the respective proof strategies.
\subsection{Notations and definitions}~\\
Before setting out the main results of this article, we briefly introduce the various notations and definitions used throughout. 
We will refer to $C>0$ a constant independent of $\varepsilon$ and of time and $f\lesssim g$ will mean $f\leq C g$. For all Banach space $X$ and all functions $f,g\in X$, we set up $\|(f,g)\|_X\mathrel{\mathop:=}\|f\|_X+\|g\|_X$. We denote by $L^2(\R_+;X)$ the set of measurable functions $f:[0,+\infty[\rightarrow X$ such that $t\mapsto \|f(t)\|_{X}$ is in $L^2(\R_+)$ and let us write $\|\cdot \|_{L^2(X)}\mathrel{\mathop:}=\|\cdot \|_{L^2(\R_+;X)}$.

We describe in the appendix the construction and properties of Besov spaces.
\subsection{Main result}
In this article, we prove the following theorem: 
\begin{thm}\label{BIG theorem}
    Let us consider the system \eqref{NS remise à l'échelle} for $\varepsilon>0$. 
    
    Then there exists a positive constant $\alpha$ (independent of $\varepsilon$) such that for all initial data $u_0=(v_0,w_0)$ where $v_0 \in \dot B_{2,1}^{\frac{1}{2}}\cap \dot B_{2,1}^{\frac{3}{2}}$ and $\overline{u}_0=(\overline{v}_0,\overline{w}_0)$ satisfying:  

    \begin{equation}\label{condition de petitesse système limite}
    \begin{aligned} \|v_0\|_{\dot B_{2,1}^{\frac{1}{2}}}+\|v_0\|_{\dot B_{2,1}^{\frac{3}{2}}} & \leq \alpha, \quad \text{and} \quad \dive u_0=0 \\
    \|\overline{v}_0\|_{\dot B_{2,1}^{\frac{1}{2}}}+\|\overline{v}_0\|_{\dot B_{2,1}^{\frac{3}{2}}} & \leq \alpha \quad \text{and} \quad \dive \overline{u}_0=0, 
    \end{aligned}
    \end{equation} with $v_0$ and $\overline{v}_0$ even (resp. $w_0$ and $\overline{w}_0$ odd) with respect to the vertical coordinate $z$, 
    the system \eqref{Equations primitives} with initial data $u_0$ admits a unique global-in-time solution $(u,p)$ with $u=(v,w)$ where $v$ is in the set $E$ defined by 
    
    \begin{eqnarray}\label{définition espace fonctionnel E} E\mathrel{\mathop:}=\mathcal{C}_b\big(\R_+;\dot B_{2,1}^{\frac{1}{2}}\cap \dot B_{2,1}^{\frac{3}{2}}\big)\cap \ L^1\big(\R_+;\dot B_{2,1}^{\frac{5}{2}}\cap \dot B_{2,1}^{\frac{7}{2}}\big),\end{eqnarray} and $\nabla_H p$ in $L^1\left(\R_+;\dot B_{2,1}^{\frac{1}{2}}\cap \dot B_{2,1}^{\frac{3}{2}}\right)$ verifying the following inequality for all $t\in\R_+$ : 
    \begin{equation}\label{estimee finale systeme limite besov}\|v(t)\|_{\dot B_{2,1}^{\frac{1}{2}}\cap \dot B_{2,1}^{\frac{3}{2}}}+\int_0^t \big(\|v\|_{\dot B_{2,1}^{\frac{5}{2}}\cap \dot B_{2,1}^{\frac{7}{2}}}+\|\nabla_H p\|_{\dot B_{2,1}^{\frac{1}{2}}\cap \dot B_{2,1}^{\frac{3}{2}}}\big) d\tau\leq C\|v_0\|_{\dot B_{2,1}^{\frac{1}{2}}\cap \dot B_{2,1}^{\frac{3}{2}} },\end{equation}
    and the system \eqref{NS remise à l'échelle} with inital data $\overline{u}_0$ admits a unique global-in-time solution $(u_\varepsilon, p_\varepsilon) $ with  $u_\varepsilon=(v_\varepsilon,w_\varepsilon)$ where $v_\varepsilon$ is in the set $E$ and $(\nabla_H, \varepsilon^{-1}\partial_z) p_\varepsilon$ in $L^1\left(\R_+;\dot B_{2,1}^{\frac{1}{2}}\cap \dot B_{2,1}^{\frac{3}{2}}\right)$ verifying for all $t\in\R_+$ : $$\displaylines{\|v_\varepsilon(t)\|_{\dot B_{2,1}^{\frac{1}{2}}\cap \dot B_{2,1}^{\frac{3}{2}}}+\int_0^t \left(\|v_\varepsilon\|_{\dot B_{2,1}^{\frac{5}{2}}\cap \dot B_{2,1}^{\frac{7}{2}}}+\|(\nabla_H, \varepsilon^{-1}\partial_z) p_\varepsilon\|_{\dot B_{2,1}^{\frac{1}{2}}\cap \dot B_{2,1}^{\frac{3}{2}}}\right) d\tau \hfill\cr\hfill \leq \|\overline{v}_0\|_{\dot B_{2,1}^{\frac{1}{2}}\cap\dot B_{2,1}^{\frac{3}{2}}}. }$$

    If, moreover, $\|\overline{v}_0-v_0\|_{\dot B_{2,1}^{\frac{1}{2}}\cap \dot B_{2,1}^{\frac{3}{2}}}\leq C\varepsilon$ then we have : \begin{eqnarray}\label{taux de convergence}\|v_\varepsilon-v\|_{L^\infty(\R_+;\dot B_{2,1}^{\frac{1}{2}}\cap \dot B_{2,1}^\frac{3}{2})\cap L^1(\R_+;\dot B_{2,1}^{\frac{5}{2}}\cap \dot B_{2,1}^\frac{7}{2})}  \lesssim \varepsilon.\end{eqnarray}
\end{thm}

\begin{remark}
    The estimate \eqref{taux de convergence} gives us the information that $w_\varepsilon$ converges weakly to $w$ in $L^\infty(\R_+; \dot B_{2,1}^{\frac{1}{2}})\cap L^1(\R^+; \dot B_{2,1}^{\frac{5}{2}})$ since we have, by Lemma \eqref{inégalité poincaré-wirtinger} and the condition of divergence free, $$\|w_\varepsilon-w\|_{\dot B_{2,1}^s}\leq \|\partial_z w_\varepsilon-w\|_{\dot B_{2,1}^s}=\|\dive_H (v_\varepsilon-v)\|_{\dot B_{2,1}^{s}}\lesssim \|v_\varepsilon-v\|_{\dot B_{2,1}^{s+1}}.$$
\end{remark}

\subsection{Sketch of the proof}~\\
We divide the proof of this result into three parts. In the first two subsections, we focus on the well-posedness of these two systems, and prove more precisely that for small enough initial data, these systems (studied in $E$) admit a unique global-in-time solution. 

In the final subsection, we prove the convergence of the solutions.

To do this, we will divide the proof of the well-posedness of the systems into three parts. The first (and most important) step is to assume that we have a regular enough solution, localize our system with the dyadic blocks and deduce the associated classical energy estimates, which are obtained by taking the scalar product in $L^2$ of the system with the localized solution and using integrations by parts and various properties of this system: we then deduce the a priori estimates. 

Once the a priori estimates are available, we use a classic approximation scheme to obtain the existence theorem for global solutions in time: this is Friedrichs' method (presented in \cite{BCD}).

For uniqueness, we look at the system verified by the difference of two solutions and derive an estimate, and end the proof of uniqueness with Grönwall's lemma.

Concerning the proof of convergence of solutions, using the fact that $\partial_z p_\varepsilon=\mathcal{O}(\varepsilon)$ in $L^1(\R_+; \dot B_{2,1}^{\frac{1}{2}})$ for the pressure, we deduce by studying the estimates verified by the difference of the two solutions of the system that we have $(v_\varepsilon,\varepsilon w_\varepsilon)-(v,w)=\mathcal{O}(\varepsilon)$ in $E$.

\section{Proof of the results}
Firstly, let us look at the study of primitive equations.
\subsection{Study of primitive equations for the ocean and atmosphere}
In this subsection, we focus on the result of Theorem \ref{BIG theorem} about the well-posedness and uniqueness of the system \eqref{Equations primitives}.

Let us begin by finding the a priori estimates \eqref{estimee finale systeme limite besov} associated to the system.
\subsubsection{A priori estimates}~\\
We assume that we have at our disposal a sufficiently regular solution of the system.

First, we will deduce from the classical energy method, an estimate on $v$. 

To do so, we apply the localization operator $\dot\Delta_j$ to the system \eqref{Equations primitives}. We get : 
\begin{eqnarray}\label{Equations primitives localisées}\left\{\begin{array}{l}
   \partial_t v_j + \dot\Delta_j\left( u\cdot \nabla v\right) - \Delta v_j +\nabla_H p_j=0,  
   \\
    \partial_z p_j =0,
    \\
    \dive_h v_j+\partial_z w_j =0.
\end{array}\right.\end{eqnarray}
By taking the product scalar with $v_j$ in the first equation of \eqref{Equations primitives localisées},  we have by integration by parts for the measure $dX=d(x,y,z)$: $$\frac{1}{2}\frac{d}{dt}\|v_j\|_{L^2}^2+\|\nabla v_j\|_{L^2}^2=-\int_{\Omega}\nabla_H p_j\cdot v_jdX+\int_\Omega \dot\Delta_j (u\cdot \nabla v)\cdot v_jdX.$$
From the last two equations of \eqref{Equations primitives localisées}, we deduce by integration by parts : \begin{equation}\label{intégrale pression nulle}-\int_\Omega \nabla_H p_j\cdot v_j dX=\int_\Omega p_j \dive_H v_j dX=-\int_\Omega p_j \partial_z w_j dX=\int_\Omega \partial_z p_j w_j dX =0.\end{equation}
By the Cauchy-Schwarz inequality, we therefore deduce : $$\frac{1}{2}\frac{d}{dt}\|v_j\|_{L^2}^2+\|\nabla v_j\|_{L^2}^2=\int_\Omega \dot\Delta_j (u\cdot \nabla v)\cdot v_jdX\leq \|\dot \Delta_j (u\cdot \nabla v)\|_{L^2}\|v_j\|_{L^2}.$$

By Bernstein's lemma (see \cite{BCD}), we have $\|\nabla v_j\|_{L^2}\simeq 2^j \|v_j\|_{L^2}$.

By Lemma \ref{lemme edo}, we then obtain : $$\|v_j(t)\|_{L^2}+c\int_0^t 2^{2j}\|v_j\|_{L^2}d\tau\leq \|v_{j,0}\|_{L^2}+\int_0^t \|\dot \Delta_j(u\cdot \nabla v)\|_{L^2}d\tau.$$

By multiplyling by $2^{js}$ with $s\in\R$ and summing up on $j\in\Z$, we then deduce : $$\|v(t)\|_{\dot B_{2,1}^{s}}+c\int_0^t \|v\|_{\dot B_{2,1}^{s+2}}d\tau\leq \|v_0\|_{\dot B_{2,1}^{s}}+\int_0^t \|u\cdot \nabla v\|_{\dot B_{2,1}^s}d\tau.$$

By using $\dive_H v+\partial_z w=0$, $w$ is odd and the Poincaré's inequality \eqref{inégalité poincaré-wirtinger}, we then deduce : \begin{eqnarray}\label{lien besov $v$ et $w$}
    \|w\|_{\dot B_{2,1}^{s}}\leq \|\partial_z w\|_{\dot B_{2,1}^{s}}= \|\dive_H v\|_{\dot B_{2,1}^{s}}\lesssim \| v\|_{\dot B_{2,1}^{s+1}}.\end{eqnarray}

Let us take $s=\frac{1}{2}$ in a first time. By the product laws of Lemma \ref{Produit espace de Besov} and by \eqref{lien besov $v$ et $w$}, we get : $$\|v\cdot \nabla_H v\|_{\dot B_{2,1}^{\frac{1}{2}}}\lesssim \|v\|_{\dot B_{2,1}^{\frac{1}{2}}}\|\nabla_H v\|_{\dot B_{2,1}^{\frac{3}{2}}}\lesssim \|v\|_{\dot B_{2,1}^{\frac{1}{2}}}\|v\|_{\dot B_{2,1}^{\frac{5}{2}}},$$ and $$\|w\partial_z v\|_{\dot B_{2,1}^{\frac{1}{2}}}\lesssim \|w\|_{\dot B_{2,1}^{\frac{1}{2}}}\|\partial_z v\|_{\dot B_{2,1}^{\frac{3}{2}}}\lesssim \|v\|_{\dot B_{2,1}^{\frac{3}{2}}}\|v\|_{\dot B_{2,1}^{\frac{5}{2}}}.$$

So we have : $$\|v(t)\|_{\dot B_{2,1}^{\frac{1}{2}}}+c\int_0^t \|v\|_{\dot B_{2,1}^{\frac{5}{2}}}d\tau \lesssim \|v_0\|_{\dot B_{2,1}^{\frac{1}{2}}}+C\int_0^t \|v\|_{\dot B_{2,1}^{\frac{1}{2}}\cap \dot B_{2,1}^{\frac{3}{2}}} \|v\|_{\dot B_{2,1}^{\frac{5}{2}}} d\tau.$$

Now taking $s=\frac{3}{2}$, we have by the product laws of Lemma \ref{Produit espace de Besov} and by \eqref{lien besov $v$ et $w$} :  $$\|v\cdot \nabla_H v\|_{\dot B_{2,1}^{\frac{3}{2}}}\lesssim \|v\|_{\dot B_{2,1}^{\frac{3}{2}}}\|\nabla_H v\|_{\dot B_{2,1}^{\frac{3}{2}}}\lesssim \|v\|_{\dot B_{2,1}^{\frac{3}{2}}}\|v\|_{\dot B_{2,1}^{\frac{5}{2}}},$$ and $$\|w\partial_z v\|_{\dot B_{2,1}^{\frac{3}{2}}}\lesssim \|w\|_{\dot B_{2,1}^{\frac{3}{2}}}\|\partial_z v\|_{\dot B_{2,1}^{\frac{3}{2}}}\lesssim \|v\|_{\dot B_{2,1}^{\frac{5}{2}}}^2.$$

So we have: $$\|v(t)\|_{\dot B_{2,1}^{\frac{3}{2}}}+c\int_0^t \|v\|_{\dot B_{2,1}^{\frac{7}{2}}}d\tau \lesssim \|v_0\|_{\dot B_{2,1}^{\frac{3}{2}}}+C\int_0^t \left(\|v\|_{\dot B_{2,1}^{\frac{3}{2}}}\|v\|_{\dot B_{2,1}^{\frac{5}{2}}} +\|v\|_{\dot B_{2,1}^{\frac{5}{2}}}^2\right) d\tau.$$

Summing up the inequalities for $s=\frac{1}{2}$ and $s=\frac{3}{2}$, we obtain : $$\displaylines{\|v(t)\|_{\dot B_{2,1}^{\frac{1}{2}}\cap \dot B_{2,1}^{\frac{3}{2}}}+c\int_0^t \|v\|_{\dot B_{2,1}^{\frac{5}{2}}\cap \dot B_{2,1}^{\frac{7}{2}}} d\tau\lesssim  \|v_0\|_{\dot B_{2,1}^{\frac{1}{2}}\cap \dot B_{2,1}^{\frac{3}{2}}} \hfill\cr\hfill +C\int_0^t \bigg(\|v\|_{\dot B_{2,1}^{\frac{1}{2}}\cap \dot B_{2,1}^{\frac{3}{2}}}\|v\|_{\dot B_{2,1}^{\frac{5}{2}}} +\|v\|_{\dot B_{2,1}^{\frac{5}{2}}}^2\bigg) d\tau.} $$

By interpolation, we have : $$\|v\|_{\dot B_{2,1}^{\frac{5}{2}}}^2\lesssim \|v\|_{\dot B_{2,1}^{\frac{3}{2}}}\|v\|_{\dot B_{2,1}^{\frac{7}{2}}}.$$

Setting $$\mathcal{A}(t)\mathrel{\mathop:}=  \|v(t)\|_{\dot B_{2,1}^{\frac{1}{2}}\cap \dot B_{2,1}^{\frac{3}{2}}}, \quad \mathcal{B}(t)\mathrel{\mathop:}= \|v(t)\|_{\dot B_{2,1}^{\frac{5}{2}}\cap \dot B_{2,1}^{\frac{7}{2}}},$$ we conclude to the following inequality : $$\mathcal{A}(t)+c\int_0^t \mathcal{B}(\tau) d\tau \leq \mathcal{A}(0)+C\int_0^t \mathcal{A}(\tau)\mathcal{B}(\tau) d\tau.  $$

Then, we have by Lemma \ref{lemme edo2} for a small initial condition : $$\mathcal{A}(t)+\frac{c}{2}\int_0^t \mathcal{B}(\tau) d\tau\leq \mathcal{A}(0).$$

Now let us estimate the pressure term.
\begin{lemma} The pressure may be defined :
\begin{eqnarray}\label{pression pour équations primitives}
\displaystyle p=\frac{1}{2}\int_{-1}^1 (-\Delta)^{-1}\dive_H(u\cdot \nabla v)dz'. 
\end{eqnarray}
Furthermore, it verifies : \begin{eqnarray}\label{pression pour équations primitives 2}\int_{-1}^1 \dive_H(\nabla_H p)dz'=-\int_{-1}^1 \dive_H (u\cdot \nabla v)dz'.\end{eqnarray}

\end{lemma}
\begin{proof}
By the periodicity on the vertical component, by the zero divergence condition on $u$ and by the first equation of \eqref{Equations primitives}, we have :
\begin{align*} 0 & =\partial_t w(x,y,1)-\Delta w(x,y,1)-\left(\partial_t w(x,y,-1)-\Delta w(x,y,-1)\right) \\ & = \int_{-1}^1 (\partial_t \partial_z w-\Delta \partial_z w ) dz' \\ &  = -\int_{-1}^{1} \dive_H (\partial_t v -\Delta v)dz' \\ & =\int_{-1}^{1} \dive_H\left(\nabla_H p+u\cdot \nabla v\right)dz' .
\end{align*}

We then obtain : $$\int_{-1}^1 \dive_H(\nabla_H p)dz'=-\int_{-1}^1 \dive_H (u\cdot \nabla v)dz'.$$

But $\partial_z p=0$, so we have $$2\Delta p=-\int_{-1}^1 \dive_H (u\cdot \nabla v)dz',$$ whence \eqref{pression pour équations primitives}.
\end{proof}
By applying the operator $\dot\Delta_j$ to \eqref{pression pour équations primitives 2}, by taking the scalar product with $p_{j}$ and by integration by parts, we have : $$2\|\nabla_H p_{j}\|_{L^2}^2=\int_{\Omega}\int_{-1}^1 \dive_H\left(\dot\Delta_j(u\cdot\nabla v)\right)dz' p_{j}dX.$$

By integration by parts and the Cauchy-Schwarz inequality, we have : $$2\|\nabla_H p_{j}\|_{L^2}^2\leq \|\dot \Delta_j (u\cdot \nabla v)\|_{L^2} \|\nabla_H p_{j}\|_{L^2}.$$
We then obtain : \begin{eqnarray} \label{estimée terme de pression }\|\nabla_H p_{j}\|_{L^2}\lesssim \|\dot \Delta_j (u\cdot \nabla v)\|_{L^2}.\end{eqnarray}

We then have the product laws, \eqref{lien besov $v$ et $w$} and by interpolation : \begin{align*} \int_0^t\|\nabla_H p\|_{\dot B_{2,1}^{\frac{1}{2}}\cap \dot B_{2,1}^{\frac{3}{2}}}d\tau & \lesssim \int_0^t\|u\|_{\dot B_{2,1}^{\frac{1}{2}}\cap \dot B_{2,1}^{\frac{3}{2}}}\|v\|_{\dot B_{2,1}^{\frac{5}{2}}}d\tau \\ &  \lesssim \int_0^t \left(\|v\|_{\dot B_{2,1}^{\frac{1}{2}}\cap \dot B_{2,1}^{\frac{3}{2}}}\|v\|_{\dot B_{2,1}^{\frac{5}{2}}}+\|v\|_{\dot B_{2,1}^{\frac{5}{2}}}^2\right) d\tau \\ &  \lesssim \int_0^t \mathcal{A}(\tau)\mathcal{B}(\tau) d\tau \\  &  \lesssim \|v_0\|_{\dot B_{2,1}^{\frac{1}{2}}\cap \dot B_{2,1}^{\frac{3}{2}}},\end{align*}
whence \eqref{estimee finale systeme limite besov}.

\subsubsection{Existence theorem}~\\
Let us study the following system : $$\partial_t v+(u\cdot \nabla v)-\Delta v+\frac{1}{2}\int_{-1}^1\nabla_H (-\Delta)^{-1}\dive_H(u\cdot \nabla v)dz=0, $$
where we used \eqref{pression pour équations primitives} for the pressure and we set up $u=(v,w)$ with $w$ defined by the formal expression : $$w\mathrel{\mathop:}=-\int_{-1}^z \dive_H(v)dz',$$ coming from $\dive_H v+\partial_z w=0$ and the imparity condition on $w$.
 
We then define the following truncation operator: \begin{eqnarray} \label{opérateur de troncature }J_n u \mathrel{\mathop :}= \sum_{|k|\leq n}\mathcal{F}_H^{-1}\left((\textbf{1}_{n^{-1}\leq |\xi_H|\leq n})\mathcal{F}_H u(\xi_H)\right)(x,y)\times\widehat{u}_ke^{i\pi kz}\end{eqnarray}
where we denote by $\mathcal{F}_H$ the Fourier transformation on $\Omega_h$. $J_n$ is in particular an orthogonal projector on $L^2$.

The Friedrichs method is then used in a similar way to that presented in \cite{RD}.

We introduce the following approximating system: $$\partial_t v+J_n(J_n u\cdot \nabla J_n v)-\Delta J_n v+\frac{1}{2}\int_{-1}^1 (-\Delta)^{-1}\dive_H J_n(J_n u\cdot \nabla J_n 
v)dz=0,$$ with initial data $J_n v_0$.

\begin{itemize}
    \item[$\bullet$] By the Cauchy-Lipschitz theorem, we have (using the spectral truncation operator) that this system admits a unique maximal solution $v_n\in \mathcal{C}^1([0,T_n[; L^2)$ with initial data (for all $n\in\N$) $J_n v_0$.
\item[$\bullet$] We have $J_n v_n=v_n$ by using the uniqueness in the previous system and so $v_n$ is solution of the system :  $$\partial_t v+J_n(u\cdot \nabla v)-\Delta v+\frac{1}{2}\int_{-1}^1 (-\Delta)^{-1}\dive_H J_n(u\cdot \nabla v)dz=0,$$ with initial data $J_n v_0$.
\item[$\bullet$] By the previous estimates, we then deduce for all $t\in [0,T_n[$ : \begin{align*} \|v_n(t)\|_{\dot B_{2,1}^{\frac{1}{2}}\cap \dot B_{2,1}^{\frac{3}{2}}} +\int_0^t \|v_n(\tau)\|_{\dot B_{2,1}^{\frac{5}{2}}\cap \dot B_{2,1}^{\frac{7}{2}}} d\tau & \lesssim \|J_n(v_0)\|_{\dot B_{2,1}^{\frac{1}{2}}\cap \dot B_{2,1}^{\frac{3}{2}}} \\ & \lesssim \|v_0\|_{\dot B_{2,1}^{\frac{1}{2}}\cap \dot B_{2,1}^{\frac{3}{2}}}.\end{align*}
By extension argument of the maximal solution, we thus have that $T^n=+\infty$. 
\end{itemize}

Especially, we have uniformly in $n\in\N$ that : $$v_n\in \mathcal{C}_b(\R_+;\dot B_{2,1}^{\frac{1}{2}}\cap\dot B_{2,1}^{\frac{3}{2}})\cap L^1(\R_+;\dot B_{2,1}^{\frac{5}{2}}\cap\dot B_{2,1}^{\frac{7}{2}}).$$ In particular, we have for all $n\in\N$, $v_n$ bounded (by interpolation) in $L^2\left(\R_+;\dot B_{2,1}^{\frac{3}{2}}\right)$. We know that $\dot B_{2,1}^{\frac{3}{2}}$ is locally compact in $L^2$.We can therefore apply Ascoli's theorem and, with diagonal extraction, show that even if we extract, the sequence of approximate solutions $(v_n)_{n\in\N}$ converge to $v$ in $L^2([0,T[;L_{loc}^2(\Omega))$. 

By classical arguments of weak compactness, continuity and properties $L^1$ in time, we have that $v$ is in $E$ defined in \eqref{définition espace fonctionnel E}.

We complete the proof of the existence part of the theorem by easily verifying that this limit is indeed a solution of the system \eqref{Equations primitives} and with the information on $p$ obtained in the a priori estimates.
\subsubsection{Uniqueness}~\\
Let $(u_1,p_1)$ and $(u_2,p_2)$ be two solutions with initial data $u_0$ where $(u_1,p_1)$ is the solution found previously, verifying the inequality \eqref{estimee finale systeme limite besov} and the smallness condition \eqref{condition de petitesse système limite}.

We then have that the system satisfied by the difference of the two solutions $\delta v\mathrel{\mathop:}=v_1-v_2$ is : \begin{eqnarray}\label{système différence}\left\{\begin{array}{l} \partial_t \delta v-\Delta \delta v+\nabla_H \delta p=-\delta u \cdot \nabla v_1-u_2 \cdot \nabla \delta v \\ \delta_z \delta p=0 \\ \dive \delta u=0.\end{array}\right.\end{eqnarray}

If we prove $u_1=u_2$, then we will have the uniqueness for $\nabla p$ thanks to expression $\nabla p=\begin{pmatrix} -\partial_t v-u\cdot\nabla v+\Delta v \\ 0 \end{pmatrix}$. 

By applying $\dot\Delta_j$ to the first equation of \eqref{système différence}, we have : 
$$\partial_t \delta v_j-\Delta \delta v_j+\nabla_H \delta p_j=-\dot\Delta_j(\delta u\cdot \nabla v_1)-\dot\Delta_j(u_2\cdot\nabla\delta v).$$

By applying the scalar product with $\delta v_j$ and as \eqref{intégrale pression nulle} to eliminate the pressure term, we then deduce : $$\displaylines{\frac{1}{2}\frac{d}{dt}\|\delta v_j\|_{L^2}^2+\|\nabla \delta v_j\|_{L^2}^2=-\int_{\Omega}\dot\Delta_j(\delta u\cdot \nabla v_1)\cdot \delta v_j dX \hfill\cr\hfill -\int_{\Omega}\dot\Delta_j(u_2\cdot \nabla\delta v)\cdot \delta v_j \ dX.}$$

We have also: $$\dot\Delta_j(u_2\cdot\nabla)\delta v=(u_2\cdot\nabla)\delta v_j+[\dot\Delta_j, u_2\cdot\nabla]\delta v.$$
By integration by parts, since $\dive u_0=0$, we get : \begin{align*} \int_{\Omega}(u_2\cdot\nabla)\delta v_j\cdot\delta v_j \ dX =0.\end{align*}

By the Cauchy-Schwarz inequality, we deduce : \begin{equation}\label{inégalité CS1}\frac{1}{2}\frac{d}{dt}\|\delta v_j\|_{L^2}^2+2^{2j}\|v_j\|_{L^2}^2\lesssim \left(\|\dot\Delta_j(\delta u\cdot\nabla v_1)\|_{L^2}+\|[\dot\Delta_j,u_2\cdot\nabla]\delta v\|_{L^2}\right)\|\delta v_j\|_{L^2}.
\end{equation}
By the commutator estimates, there is a sequence $(c_j)_{j\in\Z}$ verifying $\sum_{j\in\Z}c_j=1$ such that : $$\|[\dot\Delta_j,u_2\cdot\nabla]\delta v\|_{L^2}\leq C c_j 2^{-\frac{j}{2}}\|\nabla u_2\|_{\dot B_{2,1}^{\frac{3}{2}}}\|\delta v_j\|_{L^2}\leq C c_j 2^{-\frac{j}{2}}\|u_2\|_{\dot B_{2,1}^{\frac{5}{2}}}\|\delta v_j\|_{L^2}.$$

By multiplying by $2^{\frac{j}{2}}$ the inequality \eqref{inégalité CS1}, by summing up on $j\in\Z$ and by integrating between $0$ and $t$, we have : 
$$\|\delta v(t)\|_{\dot B_{2,1}^{\frac{1}{2}}}+\int_0^t \|\delta v\|_{\dot B_{2,1}^{\frac{5}{2}}}d\tau \lesssim \int_0^t \|\delta u\cdot \nabla v_1\|_{\dot B_{2,1}^{\frac{1}{2}}} d\tau+\int_0^t \|u_2\|_{\dot B_{2,1}^{\frac{5}{2}}}\|\delta v\|_{\dot B_{2,1}^{\frac{1}{2}}} d\tau.  $$

By \eqref{lien besov $v$ et $w$}, we have $$\|u_2\|_{\dot B_{2,1}^{\frac{5}{2}}}\lesssim \|v_2\|_{\dot B_{2,1}^{\frac{5}{2}}}+\|w_2\|_{\dot B_{2,1}^{\frac{5}{2}}}\lesssim \|v_2\|_{\dot B_{2,1}^{\frac{5}{2}}\cap\dot B_{2,1}^{\frac{7}{2}}}.$$

By product laws \eqref{Produit espace de Besov} and the inequality \eqref{lien besov $v$ et $w$}, we have : \begin{align*}\|\delta u\cdot \nabla v_1\|_{\dot B_{2,1}^{\frac{1}{2}}} & \lesssim \|\delta v\cdot \nabla_H  v_1\|_{\dot B_{2,1}^{\frac{1}{2}}}+\|\delta w \partial_z v_1\|_{\dot B_{2,1}^{\frac{1}{2}}} \\ & \lesssim \|\delta v\|_{\dot B_{2,1}^{\frac{1}{2}}}\|v_1\|_{\dot B_{2,1}^{\frac{5}{2}}}+\|\delta w\|_{\dot B_{2,1}^{\frac{3}{2}}}\|v_1\|_{\dot B_{2,1}^{\frac{3}{2}}} \\ & \lesssim \|\delta v\|_{B_{2,1}^\frac{1}{2}}\|v_1\|_{\dot B_{2,1}^{\frac{5}{2}}}+\|\delta v\|_{\dot B_{2,1}^{\frac{5}{2}}}\|v_1\|_{\dot B_{2,1}^{\frac{3}{2}}}.\end{align*}

By the smallness of $\|v_1\|_{\dot B_{2,1}^{\frac{3}{2}}}$, we then deduce : $$\|\delta v(t)\|_{\dot B_{2,1}^{\frac{1}{2}}}+\int_0^t \|\delta v\|_{\dot B_{2,1}^{\frac{5}{2}}} d\tau\lesssim \int_0^t \|\delta v\|_{\dot B_{2,1}^{\frac{1}{2}}}\left(\|v_1\|_{\dot B_{2,1}^{\frac{5}{2}}}+\|v_2\|_{\dot B_{2,1}^{\frac{5}{2}}\cap \dot B_{2,1}^{\frac{7}{2}}}\right)d\tau.$$

Because $t\mapsto \|v_1(t)\|_{\dot B_{2,1}^{\frac{5}{2}}}+\|v_2(t)\|_{\dot B_{2,1}^{\frac{5}{2}}\cap \dot B_{2,1}^{\frac{7}{2}}}$ is in $L^1(\R^+)$, we then have by Grönwall's lemma : $$\|\delta v(t)\|_{\dot B_{2,1}^{\frac{1}{2}}}=0 \quad \forall t\in\R^+.$$

\subsection{Anisotropic Navier Stokes equations}~\\
The system \eqref{NS remise à l'échelle} can be rewritten like : 
\begin{eqnarray}\label{systeme NSA 2}\left\{\begin{array}{l}
     \partial_t \begin{pmatrix}
         v_{\varepsilon} \\ \varepsilon w_{\varepsilon}
     \end{pmatrix} +\nabla_\varepsilon p_{\varepsilon}-\Delta\begin{pmatrix}
          v_{\varepsilon} \\ \varepsilon w_{\varepsilon}
     \end{pmatrix}= \begin{pmatrix}
         -u_\varepsilon\cdot \nabla v_\varepsilon \\ -u_\varepsilon\cdot \nabla (\varepsilon w_\varepsilon)
     \end{pmatrix} \\ 
     \dive_\varepsilon (v_\varepsilon,\varepsilon w_\varepsilon)=0
     
\end{array} \right.\end{eqnarray} 

where $\dive_\varepsilon$ is defined by : \begin{eqnarray}
\dive_\varepsilon U\mathrel{\mathop:}=\dive_H(U_1,U_2)+\varepsilon^{-1}\partial_z U_3
\end{eqnarray} and $\displaystyle \nabla_\varepsilon$ by \begin{eqnarray}
\nabla_\varepsilon\mathrel{\mathop:}=\begin{pmatrix}
    \nabla_H \\ \varepsilon^{-1} \partial_z \end{pmatrix}. \end{eqnarray}

In the rest of this section we will prove the result of well-posedness and uniqueness of \eqref{NS remise à l'échelle} presented in Theorem \ref{BIG theorem}.

Let us start by proving a priori estimates for this system : 
\subsubsection{A priori estimates}
By applying $\dot \Delta_j$ to \eqref{NS remise à l'échelle}, we obtain : \begin{eqnarray}\label{systeme NSA localisé}\left\{\begin{array}{l}
     \partial_t \begin{pmatrix}
         v_{\varepsilon,j} \\ \varepsilon w_{\varepsilon,j}
     \end{pmatrix} +\nabla_\varepsilon p_{\varepsilon,j}-\Delta\begin{pmatrix}
          v_{\varepsilon,j} \\ \varepsilon w_{\varepsilon,j}
     \end{pmatrix}= \dot \Delta_j \begin{pmatrix}
         u_\varepsilon\cdot \nabla v_\varepsilon \\ u_\varepsilon\cdot \nabla (\varepsilon w_\varepsilon)
     \end{pmatrix} \\ 
     \dive u_{\varepsilon,j}=0.
     
\end{array} \right.\end{eqnarray}

Let us start by looking at the pressure term : 

By applying $\dive_\varepsilon$ to the system \eqref{systeme NSA localisé}, we obtain : $$\partial_t \dive u_{\varepsilon,j}+\Delta_\varepsilon p_{\varepsilon,j}-\Delta \dive u_{\varepsilon,j}=\dive_\varepsilon \dot \Delta_j\begin{pmatrix}
         u_\varepsilon\cdot \nabla v_\varepsilon \\ u_\varepsilon\cdot \nabla (\varepsilon w_\varepsilon)
     \end{pmatrix}, $$ where $\Delta_\varepsilon\mathrel{\mathop:}=\dive_\varepsilon \nabla_\varepsilon$.

As $\displaystyle \dive u_{\varepsilon,j}=0$, we deduce : $$\Delta_\varepsilon p_{\varepsilon,j}=\dive_\varepsilon \dot \Delta_j \begin{pmatrix}
         u_\varepsilon\cdot \nabla v_\varepsilon \\ u_\varepsilon\cdot \nabla (\varepsilon w_\varepsilon)
     \end{pmatrix}.$$
So we have: \begin{eqnarray}\label{expression de p epsilon} \nabla_\varepsilon p_{\varepsilon,j}=-\nabla_\varepsilon (-\Delta_\varepsilon)^{-1} \dive_\varepsilon \dot \Delta_j\begin{pmatrix}
         u_\varepsilon\cdot \nabla v_\varepsilon \\ u_\varepsilon\cdot \nabla (\varepsilon w_\varepsilon)
     \end{pmatrix}.\end{eqnarray}
\begin{lemma}\label{continuité opérateur de Leray}~\\
    The operator $-\nabla_\varepsilon(-\Delta_\varepsilon)^{-1}\dive_\varepsilon$ is an orthogonal projector on $L^2$.
\end{lemma}

\begin{proof}
~\\
    Let $u\in L^2$, we have : $$\mathcal{F}\left(\nabla_\varepsilon(-\Delta_\varepsilon)^{-1}\dive_\varepsilon u\right)=\frac{1}{|\xi_H|^2+\varepsilon^{-2}\xi_z^2}\begin{pmatrix}
    i\xi_H (i\xi_H\cdot \widehat{v} +\varepsilon^{-1}i\xi_z \widehat{w}) \\ \varepsilon^{-1}i\xi_z (i\xi_H\cdot \widehat{v} +\varepsilon^{-1}i\xi_z \widehat{w})
\end{pmatrix}\cdotp$$

By using Cauchy-Schwarz inequality with the variable $(\xi_H,\varepsilon^{-1}\xi_z)$, we obtain: \begin{align*}\frac{1}{|\xi_H|^2+\varepsilon^{-2}\xi_z^2}\left| i\xi_H (i\xi_H\cdot \widehat{v} +\varepsilon^{-1}i\xi_z \widehat{w})\right| \leq |\widehat{u}|
\end{align*}

and in the same way $$\frac{1}{|\xi_H|^2+\varepsilon^{-2}\xi_z^2}\left| i\varepsilon^{-1}\xi_z (i\xi_H\cdot \widehat{v} +\varepsilon^{-1}i\xi_z \widehat{w})\right| \leq |\widehat u|.$$
\end{proof}

By mutliplying by $2^{js}$ with $s\in\R$ and summing up on $j\in\Z$, we obtain : $$\|\nabla_\varepsilon p_{\varepsilon}\|_{\dot B_{2,1}^s}\leq  \|u_\varepsilon\cdot \nabla v_\varepsilon\|_{\dot B_{2,1}^s} + \|u_\varepsilon\cdot \nabla (\varepsilon w_\varepsilon)\|_{\dot B_{2,1}^s}.$$

Now let us take a look at the estimates for $v_\varepsilon$.

By taking the scalar product with $v_{\varepsilon,j}$ in the first equation of \eqref{systeme NSA localisé}, by Cauchy-Schwarz inequality, by Lemma \ref{lemme edo}, by multiplying by $2^{js}$ (with $s\in\R$) and summing up on $j\in \Z$, we obtain : $$\displaylines{\|v_\varepsilon(t)\|_{\dot B_{2,1}^s}+\int_0^t \left(\|v_\varepsilon\|_{\dot B_{2,1}^{s+2}}+\|\nabla_\varepsilon p_\varepsilon\|_{\dot B_{2,1}^s}\right)d\tau \lesssim \|\overline{v}_0\|_{\dot B_{2,1}^s} \hfill\cr\hfill +\int_0^t \|u_\varepsilon\cdot \nabla (v_\varepsilon, \varepsilon w_\varepsilon)\|_{\dot B_{2,1}^{s}}d\tau.}$$ 

By the previous esimate with $s\in\{\frac{1}{2},\frac{3}{2}\}$, we then have : 
$$\displaylines{\|v_\varepsilon(t)\|_{\dot B_{2,1}^{\frac{1}{2}}\cap \dot B_{2,1}^{\frac{3}{2}}}+\int_0^t \left(\|v_\varepsilon\|_{\dot B_{2,1}^{\frac{5}{2}}\cap \dot B_{2,1}^{\frac{7}{2}}}+\|\nabla_\varepsilon p_\varepsilon\|_{\dot B_{2,1}^{\frac{1}{2}}\cap \dot B_{2,1}^{\frac{3}{2}}}\right)d\tau  \hfill\cr\hfill \lesssim \|\overline{v}_0\|_{\dot B_{2,1}^{\frac{1}{2}}\cap \dot B_{2,1}^{\frac{3}{2}}}+\int_0^t \|u_\varepsilon\cdot \nabla (v_\varepsilon, \varepsilon w_\varepsilon)\|_{\dot B_{2,1}^{\frac{1}{2}}\cap \dot B_{2,1}^{\frac{3}{2}}}d\tau.}$$ 

Let us now consider all the non-linear terms on the right-hand side.
\begin{lemma}\label{lemme terme non linéaire 1/2}
    We have : $$\left\{\begin{array}{l} \displaystyle\|u_\varepsilon\cdot \nabla v_\varepsilon\|_{\dot B_{2,1}^{\frac{1}{2}}}\lesssim \|v_\varepsilon\|_{\dot B_{2,1}^{\frac{1}{2}}\cap \dot B_{2,1}^{\frac{3}{2}}}\|v_\varepsilon\|_{\dot B_{2,1}^{\frac{5}{2}}}, \\ \displaystyle
    \|u_\varepsilon\cdot \nabla v_\varepsilon\|_{\dot B_{2,1}^{\frac{3}{2}}}\lesssim \|v_\varepsilon\|_{\dot B_{2,1}^{\frac{3}{2}}}\|v_\varepsilon\|_{\dot B_{2,1}^{\frac{5}{2}}\cap \dot B_{2,1}^{\frac{7}{2}}}, 
    \\ \displaystyle \|u_\varepsilon\cdot \nabla (\varepsilon w_\varepsilon)\|_{\dot B_{2,1}^{\frac{1}{2}}\cap \dot B_{2,1}^{\frac{3}{2}}}\lesssim \varepsilon\|v_\varepsilon\|_{\dot B_{2,1}^{\frac{1}{2}}\cap \dot B_{2,1}^{\frac{3}{2}}}\|v_\varepsilon\|_{\dot B_{2,1}^{\frac{5}{2}}\cap \dot B_{2,1}^{\frac{7}{2}}}.
    \end{array}\right.$$

\end{lemma}

\begin{proof}
We have in a first time (like for \eqref{lien besov $v$ et $w$}) : \begin{eqnarray}\label{lien besov v_epsilon et w_epsilon}
    \|w_\varepsilon\|_{\dot B_{2,1}^{s}}\leq \|\partial_z w_\varepsilon\|_{\dot B_{2,1}^{s}}= \|\dive_H v_\varepsilon\|_{\dot B_{2,1}^{s}}\lesssim \| v_\varepsilon\|_{\dot B_{2,1}^{s+1}}.\end{eqnarray}

By product laws and \eqref{lien besov v_epsilon et w_epsilon}, we have : 

$$\|u_\varepsilon\cdot \nabla v_\varepsilon\|_{\dot B_{2,1}^{\frac{1}{2}}}\lesssim \|u_\varepsilon\|_{\dot B_{2,1}^{\frac{1}{2}}}\|\nabla v_\varepsilon\|_{\dot B_{2,1}^{\frac{3}{2}}}\lesssim \|v_\varepsilon\|_{\dot B_{2,1}^{\frac{1}{2}}\cap \dot B_{2,1}^{\frac{3}{2}}}\|v_\varepsilon\|_{\dot B_{2,1}^{\frac{5}{2}}}.$$

We have also : 
$$\|u_\varepsilon \cdot \nabla v_\varepsilon\|_{\dot B_{2,1}^{\frac{3}{2}}}\lesssim \|u_\varepsilon\|_{\dot B_{2,1}^\frac{3}{2}}\| \nabla v_\varepsilon\|_{\dot B_{2,1}^{\frac{3}{2}}}\lesssim \|v_\varepsilon\|_{\dot B_{2,1}^{\frac{3}{2}}}\|v_\varepsilon\|_{\dot B_{2,1}^{\frac{5}{2}}}+\|v_\varepsilon\|_{\dot B_{2,1}^{\frac{5}{2}}}^2.$$

By interpolation, we have : $$\|v_\varepsilon\|_{\dot B_{2,1}^{\frac{5}{2}}}^2\lesssim \|v_\varepsilon\|_{\dot B_{2,1}^{\frac{3}{2}}}\|v_\varepsilon\|_{\dot B_{2,1}^{\frac{7}{2}}}.$$

We have also : $$\displaylines{\|u_\varepsilon\cdot \nabla(\varepsilon w_\varepsilon)\|_{\dot B_{2,1}^{\frac{1}{2}}} \lesssim \|u_\varepsilon\|_{\dot B_{2,1}^{\frac{1}{2}}}\|\nabla\varepsilon w_\varepsilon\|_{\dot B_{2,1}^{\frac{3}{2}}}\lesssim \varepsilon\|v_\varepsilon\|_{\dot B_{2,1}^{\frac{1}{2}}\cap \dot B_{2,1}^{\frac{3}{2}}}\|w_\varepsilon\|_{\dot B_{2,1}^{\frac{5}{2}}} \hfill\cr\hfill \lesssim \varepsilon \|v_\varepsilon\|_{\dot B_{2,1}^{\frac{1}{2}}\cap \dot B_{2,1}^{\frac{3}{2}}}\|v_\varepsilon\|_{\dot B_{2,1}^{\frac{5}{2}}\cap \dot B_{2,1}^{\frac{7}{2}}}.}$$

 Noting that $u_\varepsilon\cdot \nabla(\varepsilon w_\varepsilon)=v_\varepsilon\cdot \nabla_H (\varepsilon w_\varepsilon)+w_\varepsilon \partial_z (\varepsilon w_\varepsilon)$, 
    we have by triangular inequality : $$\|u_\varepsilon\cdot \nabla(\varepsilon w_\varepsilon)\|_{\dot B_{2,1}^{\frac{3}{2}}}\leq \|v_\varepsilon\cdot \nabla_H (\varepsilon w_\varepsilon)\|_{\dot B_{2,1}^{\frac{3}{2}}}+\|w_\varepsilon\partial_z (\varepsilon w_\varepsilon)\|_{\dot B_{2,1}^{\frac{3}{2}}}.$$

    We have by product laws and \eqref{lien besov v_epsilon et w_epsilon} : $$\|v_\varepsilon\cdot\nabla_H (\varepsilon w_\varepsilon)\|_{\dot B_{2,1}^{\frac{3}{2}}}\lesssim \varepsilon\|v_\varepsilon\|_{\dot B_{2,1}^{\frac{3}{2}}}\|\nabla_H w_\varepsilon\|_{\dot B_{2,1}^{\frac{3}{2}}}\lesssim\varepsilon \|v_\varepsilon\|_{\dot B_{2,1}^{\frac{3}{2}}} \|v_\varepsilon\|_{\dot B_{2,1}^{\frac{7}{2}}}.$$

    We obtain also : $$\|w_\varepsilon\partial_z (\varepsilon w_\varepsilon)\|_{\dot B_{2,1}^{\frac{3}{2}}}\lesssim \varepsilon\|w_\varepsilon\|_{\dot B_{2,1}^{\frac{3}{2}}}\|\partial_z w_\varepsilon\|_{\dot B_{2,1}^{\frac{3}{2}}}\lesssim\varepsilon \|v_{\varepsilon}\|_{\dot B_{2,1}^{\frac{5}{2}}}^2\lesssim \varepsilon\|v_\varepsilon\|_{\dot B_{2,1}^{\frac{3}{2}}}\|v_\varepsilon\|_{\dot B_{2,1}^{\frac{7}{2}}}.$$

This leads to the lemma.
\end{proof}

We obtain : 
$$\displaylines{\|v_\varepsilon(t)\|_{\dot B_{2,1}^{\frac{1}{2}}\cap \dot B_{2,1}^{\frac{3}{2}}}+\int_0^t (\|v_\varepsilon\|_{\dot B_{2,1}^{\frac{5}{2}}\cap \dot B_{2,1}^{\frac{7}{2}}}+\|\nabla_\varepsilon p_\varepsilon\|_{\dot B_{2,1}^{\frac{1}{2}}\cap \dot B_{2,1}^{\frac{3}{2}}})d\tau \hfill\cr\hfill \lesssim \|\overline{v}_0\|_{\dot B_{2,1}^{\frac{1}{2}}\cap \dot B_{2,1}^{\frac{3}{2}}}+\int_0^t \|v_\varepsilon\|_{\dot B_{2,1}^{\frac{1}{2}}\cap \dot B_{2,1}^{\frac{3}{2}}}\|v_\varepsilon\|_{\dot B_{2,1}^{\frac{5}{2}}\cap \dot B_{2,1}^{\frac{7}{2}}} d\tau.}$$

By Lemma \ref{lemme edo2}, we get for all $t\in[0,T]$ :
$$\displaylines{\|v_\varepsilon(t)\|_{\dot B_{2,1}^{\frac{1}{2}}\cap \dot B_{2,1}^{\frac{3}{2}}}+\int_0^t (\|v_\varepsilon\|_{\dot B_{2,1}^{\frac{5}{2}}\cap \dot B_{2,1}^{\frac{7}{2}}}+\|\nabla_\varepsilon p_\varepsilon\|_{\dot B_{2,1}^{\frac{1}{2}}\cap \dot B_{2,1}^{\frac{3}{2}}})d\tau \lesssim \|\overline{v}_0\|_{\dot B_{2,1}^{\frac{1}{2}}\cap \dot B_{2,1}^{\frac{3}{2}}}.}$$

Hence the final a priori estimate of the theorem. 
\subsubsection{Existence theorem}~\\
To remove the pressure term, we do as in the classical case (without anisotropy) where we use the Leray projector. Here, the latter is slightly modified by the anisotropy, but the continuity properties remain the same. Let's consider the anisotropic Leray projector: $$\mathbb{P}_\varepsilon\mathrel{\mathop:}=Id+\nabla_\varepsilon (-\Delta_\varepsilon)^{-1}\dive_{\varepsilon},$$ this expression coming from \eqref{expression de p epsilon}.

In particular, it is a continuous operator with norm 1 from $\dot B_{2,1}^{s}$ to $\dot B_{2,1}^{s}$ for all $s\in\R$ by Lemma \ref{continuité opérateur de Leray} wich satisfies $\mathbb{P}_\varepsilon (v,\varepsilon w)=(v,\varepsilon w)$ for $u=(v,w)$ with $v\in \dot B_{2,1}^{\frac{1}{2}}\cap\dot B_{2,1}^{\frac{3}{2}}$ verifying $\dive_\varepsilon u=0$. 
Finding solutions $\left((v_\varepsilon,\varepsilon w_\varepsilon),p_\varepsilon\right)$ in the system \eqref{NS remise à l'échelle} with initial data $\overline{u}_0$ is equivalent to finding solutions $(v_\varepsilon,\varepsilon w_\varepsilon)$ to the following system with initial condition $\mathbb{P}_\varepsilon \overline{u}_0$ : \begin{eqnarray}\label{NS projeté} \partial_t \begin{pmatrix}
    v_\varepsilon \\ \varepsilon w_\varepsilon
\end{pmatrix}-\Delta \begin{pmatrix}
    v_\varepsilon \\ \varepsilon w_\varepsilon
\end{pmatrix}=-\mathbb{P}_\varepsilon \begin{pmatrix}
    u_\varepsilon\cdot \nabla v_\varepsilon \\ u_\varepsilon\cdot \nabla(\varepsilon w_\varepsilon)
\end{pmatrix}.\end{eqnarray} 
To obtain the existence theorem after obtaining the a priori estimates, we argue using Friedrichs' method like previously.

\subsubsection{Uniqueness}~\\
Let $(u_{\varepsilon,1},p_{\varepsilon,1})$ and $(u_{\varepsilon,2},p_{\varepsilon,2})$ be two solutions of \eqref{NS remise à l'échelle} with initial data $\overline{u}_0$.

The system satisfied by the difference between the two solutions $\delta u\mathrel{\mathop:}=u_{\varepsilon,1}-u_{\varepsilon,2}$, $\delta p_\varepsilon\mathrel{\mathop:}=p_{\varepsilon,1}-p_{\varepsilon,2}$ is : 
$$\left\{\begin{array}{l}
     \partial_t \delta v_\varepsilon-\Delta \delta v_\varepsilon+\nabla_H \delta p_\varepsilon=-\delta u_\varepsilon \cdot \nabla v_{\varepsilon,1}-u_{\varepsilon,2}\cdot \nabla \delta v_\varepsilon \\ \displaystyle
      \partial_t (\varepsilon \delta w_\varepsilon)-\Delta \varepsilon \delta w_\varepsilon+\frac{\partial_z \delta p_\varepsilon}{\varepsilon}=-\varepsilon \delta u_\varepsilon\cdot \nabla w_{\varepsilon,1}-u_{\varepsilon,2}\cdot \nabla(\varepsilon \delta w_\varepsilon)
      \\ \dive \delta u_\varepsilon=0.
\end{array}\right.$$

By applying $\mathbb{P}_\varepsilon$, we get : 
$$\frac{d}{dt} \begin{pmatrix}\delta v_\varepsilon \\ \varepsilon \delta w_\varepsilon \end{pmatrix}-\Delta \begin{pmatrix} \delta v_\varepsilon \\ \varepsilon \delta w_\varepsilon   \end{pmatrix}=-\mathbb{P}_\varepsilon\begin{pmatrix} \delta u_\varepsilon \cdot \nabla v_{\varepsilon,1}+u_{\varepsilon,2}\cdot \nabla \delta v_\varepsilon \\ \varepsilon \delta u_\varepsilon\cdot \nabla w_{\varepsilon,1}+u_{\varepsilon,2}\cdot \nabla(\varepsilon \delta w_\varepsilon)\end{pmatrix}
.$$

By applying $\dot\Delta_j$, we can rewrite the system as follows :
$$\displaylines{\frac{d}{dt} \begin{pmatrix}\delta v_{\varepsilon,j} \\ \varepsilon \delta w_{\varepsilon,j} \end{pmatrix}-\Delta \begin{pmatrix} \delta v_{\varepsilon,j} \\ \varepsilon \delta w_{\varepsilon,j}   \end{pmatrix} \hfill\cr\hfill =-\mathbb{P}_\varepsilon\begin{pmatrix} \dot\Delta_j(\delta u_\varepsilon \cdot \nabla v_{\varepsilon,1})+u_{\varepsilon,2}\cdot \nabla \delta v_{\varepsilon,j}+[\dot\Delta_j, u_{\varepsilon,2}\cdot\nabla]\delta v_\varepsilon \\ \varepsilon \dot\Delta_j(\delta u_\varepsilon\cdot \nabla w_{\varepsilon,1})+u_{\varepsilon,2}\cdot \nabla(\varepsilon \delta w_{\varepsilon,j})+[\dot\Delta_j,u_{\varepsilon,2}\cdot\nabla](\varepsilon\delta w_{\varepsilon})\end{pmatrix}
.}$$

Taking the scalar product with $(\delta v_{\varepsilon,j},\varepsilon \delta w_{\varepsilon,j})$ and by Cauchy-Schwarz inequality, we obtain :

$$\displaylines{\frac{1}{2}\frac{d}{dt}\|(\delta v_{\varepsilon,j},\varepsilon \delta w_{\varepsilon,j})\|_{L^2}^2+2^{2j}\|(\delta v_{\varepsilon,j},\varepsilon \delta w_{\varepsilon,j})\|_{L^2}^2 \hfill\cr \lesssim \bigg(\|\dot\Delta_j (\delta u_\varepsilon\cdot \nabla v_{\varepsilon,1})\|_{L^2}+\|[\dot\Delta_j, u_{\varepsilon,2}\cdot \nabla]\delta v_\varepsilon\|_{L^2}+\varepsilon\|\dot \Delta_j (\delta u_\varepsilon\cdot \nabla w_{\varepsilon,1})\|_{L^2} \hfill\cr\hfill+\|[\dot\Delta_j,u_{\varepsilon,2}\cdot\nabla](\varepsilon \delta w_\varepsilon)\|_{L^2}\bigg)\|(\delta v_{\varepsilon,j},\varepsilon w_{\varepsilon,j})\|_{L^2}}$$

However, we have by the product laws of Lemma \ref{Produit espace de Besov} and by \eqref{lien besov v_epsilon et w_epsilon} : 
\begin{align*} \|\delta u_{\varepsilon}\cdot \nabla v_{\varepsilon,1}\|_{\dot B_{2,1}^{\frac{1}{2}}} & \lesssim \|\delta v_{\varepsilon}\cdot \nabla_H v_{\varepsilon,1}\|_{\dot B_{2,1}^{\frac{1}{2}}}+\|\delta w_{\varepsilon} \partial_z v_{\varepsilon,1}\|_{\dot B_{2,1}^{\frac{1}{2}}} \\ & \lesssim \|\delta v_{\varepsilon} \|_{\dot B_{2,1}^{\frac{1}{2}}}\|v_{\varepsilon,1}\|_{\dot B_{2,1}^{\frac{5}{2}}}+\|\delta w_\varepsilon\|_{\dot B_{2,1}^{\frac{3}{2}}}\|v_{\varepsilon,1}\|_{\dot B_{2,1}^{\frac{3}{2}}}
 \\ & \lesssim \|\delta v_{\varepsilon} \|_{\dot B_{2,1}^{\frac{1}{2}}}\|v_{\varepsilon,1}\|_{\dot B_{2,1}^{\frac{5}{2}}}+\|\delta v_\varepsilon\|_{\dot B_{2,1}^{\frac{5}{2}}}\|v_{\varepsilon,1}\|_{\dot B_{2,1}^{\frac{3}{2}}}
\end{align*}
and
\begin{align*}
    \varepsilon\|\delta u_{\varepsilon}\cdot \nabla w_{\varepsilon,1}\|_{\dot B_{2,1}^{\frac{1}{2}}} & \lesssim \varepsilon\|\delta v_{\varepsilon}\cdot \nabla_H w_{\varepsilon,1}\|_{\dot B_{2,1}^{\frac{1}{2}}}+\varepsilon\|\delta w_{\varepsilon}\partial_z w_{\varepsilon,1}\|_{\dot B_{2,1}^{\frac{1}{2}}} \\ & \lesssim \varepsilon\|\delta v_{\varepsilon} \|_{\dot B_{2,1}^{\frac{1}{2}}}\|w_{\varepsilon,1}\|_{\dot B_{2,1}^{\frac{5}{2}}}+\varepsilon\|\delta w_\varepsilon\|_{\dot B_{2,1}^{\frac{3}{2}}}\|\partial_z w_{\varepsilon,1}\|_{\dot B_{2,1}^{\frac{1}{2}}}
 \\ & \lesssim \varepsilon\|\delta v_{\varepsilon} \|_{\dot B_{2,1}^{\frac{1}{2}}}\|v_{\varepsilon,1}\|_{\dot B_{2,1}^{\frac{7}{2}}}+\varepsilon\|\delta v_\varepsilon\|_{\dot B_{2,1}^{\frac{5}{2}}}\|v_{\varepsilon,1}\|_{\dot B_{2,1}^{\frac{3}{2}}}.
\end{align*}
By commutator estimates, there exists a sequence $(c_j)_{j\in\Z}$ such that
$$\sum_{j\in\Z}c_j=1$$ and which verifies : 
\begin{align*} \|[\dot\Delta_j,u_{\varepsilon,2}\cdot\nabla](\delta v_\varepsilon, \varepsilon\delta w_\varepsilon)\|_{L^2}  \leq C c_j 2^{-\frac{j}{2}}\|u_{\varepsilon,2}\|_{\dot B_{2,1}^{\frac{5}{2}}}\|(\delta v_\varepsilon,\varepsilon \delta w_\varepsilon)\|_{\dot B_{2,1}^{\frac{1}{2}}}.
\end{align*}

By multiplying by $2^{\frac{j}{2}}$, summing up on $j\in\Z$ and integrating between $0$ and $t$, we then deduce : 
$$\displaylines{\|(\delta v_\varepsilon,\varepsilon \delta w_\varepsilon)(t)\|_{\dot B_{2,1}^{\frac{1}{2}}}+\int_0^t \|(\delta v_\varepsilon,\varepsilon \delta w_\varepsilon)(t)\|_{\dot B_{2,1}^{\frac{5}{2}}} d\tau \hfill\cr\lesssim \int_0^t (\|u_{\varepsilon,2}\|_{\dot B_{2,1}^{\frac{5}{2}}}+\|v_{\varepsilon,1}\|_{\dot B_{2,1}^{\frac{5}{2}}\cap\dot B_{2,1}^{\frac{7}{2}}})\|(\delta v_\varepsilon,\varepsilon \delta w_\varepsilon)\|_{\dot B_{2,1}^{\frac{1}{2}}} d\tau \hfill\cr\hfill+\int_0^t \|\delta v_\varepsilon\|_{\dot B_{2,1}^{\frac{5}{2}}}\|v_{\varepsilon,1}\|_{\dot B_{2,1}^{\frac{3}{2}}}  d\tau.}  $$

By smallness of $\|v_{\varepsilon,1}\|_{\dot B_{2,1}^{\frac{3}{2}}}$, we then deduce : 
$$\displaylines{\|(\delta v_\varepsilon,\varepsilon \delta w_\varepsilon)(t)\|_{\dot B_{2,1}^{\frac{1}{2}}}+\int_0^t \|(\delta v_\varepsilon,\varepsilon \delta w_\varepsilon)(t)\|_{\dot B_{2,1}^{\frac{5}{2}}} d\tau \hfill\cr\lesssim \int_0^t (\|u_{\varepsilon,2}\|_{\dot B_{2,1}^{\frac{5}{2}}}+\|v_{\varepsilon,1}\|_{\dot B_{2,1}^{\frac{5}{2}}\cap\dot B_{2,1}^{\frac{7}{2}}})\|(\delta v_\varepsilon,\varepsilon \delta w_\varepsilon)\|_{\dot B_{2,1}^{\frac{1}{2}}} d\tau.}  $$

By Grönwall's lemma and the fact that $t\mapsto \|u_{\varepsilon,2}(t)\|_{\dot B_{2,1}^{\frac{5}{2}}}+\|v_{\varepsilon,1}(t)\|_{\dot B_{2,1}^{\frac{5}{2}}\cap \dot B_{2,1}^{\frac{7}{2}}}$ is in  $L^1(\R^+)$, we then have : $$\forall t\in\R^+, \quad \|(\delta v_\varepsilon,\varepsilon \delta w_\varepsilon)(t)\|_{\dot B_{2,1}^{\frac{1}{2}}}=0,$$ 
whence uniqueness.
\subsection{Passing to the limit between the two systems}~\\
We want to study the equation verified by the difference between the solutions $(v_\varepsilon,w_\varepsilon)$ of \eqref{NS remise à l'échelle} and that of the primitive equation \eqref{Equations primitives} for $(v,w)$. We set up $$(V_\varepsilon,\varepsilon W_\varepsilon)\mathrel{\mathop:}=(v_\varepsilon-v,\varepsilon(w_\varepsilon-w)), \ U_\varepsilon\mathrel{\mathop:}=(V_\varepsilon,W_\varepsilon), \ P_\varepsilon\mathrel{\mathop:}=p_\varepsilon-p.$$
The system satisfied by $(V_\varepsilon,\varepsilon W_\varepsilon)$ is : 
\begin{eqnarray}\label{système principal}
    \left\{\begin{array}{l}
    \partial_t V_\varepsilon-\Delta V_\varepsilon+\nabla_H P_\varepsilon= -U_\varepsilon\cdot \nabla v-u_\varepsilon\cdot \nabla V_\varepsilon, \\
    \partial_t(\varepsilon W_\varepsilon)-\Delta(\varepsilon W_\varepsilon)+\frac{1}{\varepsilon}\partial_z P_\varepsilon =\varepsilon F(U_\varepsilon,u_\varepsilon,u),
    \end{array}
    \right.
\end{eqnarray} where $F(U_\varepsilon,u_\varepsilon,u)=-U_\varepsilon\cdot \nabla w-u_\varepsilon\cdot \nabla W_\varepsilon-\partial_t w-u\cdot \nabla w+\Delta w$.

With the help of \eqref{système principal} , we have : $$ \displaylines{\|V_\varepsilon(t)\|_{\dot B_{2,1}^{\frac{1}{2}}\cap \dot B_{2,1}^\frac{3}{2}}+\int_0^t \|V_\varepsilon\|_{\dot B_{2,1}^{\frac{5}{2}}\cap \dot B_{2,1}^\frac{7}{2}} d\tau \hfill\cr  \lesssim \|\overline{v}_0-v_0\|_{\dot B_{2,1}^{\frac{1}{2}}\cap \dot B_{2,1}^{\frac{3}{2}}} +\int_0^t \|U_\varepsilon\cdot \nabla v\|_{\dot B_{2,1}^{\frac{1}{2}}\cap \dot B_{2,1}^\frac{3}{2}} d\tau +\int_0^t \|u_\varepsilon\cdot \nabla V_\varepsilon\|_{\dot B_{2,1}^{\frac{1}{2}}\cap \dot B_{2,1}^\frac{3}{2}} d\tau.}$$

By the same calculations as in the proof of Lemma \ref{lemme terme non linéaire 1/2}, we have : $$\displaylines{\int_0^t \|U_\varepsilon\cdot \nabla v\|_{\dot B_{2,1}^{\frac{1}{2}}\cap \dot B_{2,1}^\frac{3}{2}} d\tau +\int_0^t \|u_\varepsilon\cdot \nabla V_\varepsilon\|_{\dot B_{2,1}^{\frac{1}{2}}\cap \dot B_{2,1}^\frac{3}{2}} d\tau \hfill\cr\hfill \lesssim \int_0^t \|V_\varepsilon\|_{\dot B_{2,1}^{\frac{1}{2}}\cap \dot B_{2,1}^\frac{3}{2}}\|v\|_{\dot B_{2,1}^{\frac{5}{2}}\cap \dot B_{2,1}^\frac{7}{2}} d\tau +\int_0^t \|u_\varepsilon\|_{\dot B_{2,1}^{\frac{1}{2}}\cap \dot B_{2,1}^\frac{3}{2}} \|V_\varepsilon\|_{\dot B_{2,1}^{\frac{5}{2}}\cap \dot B_{2,1}^\frac{7}{2}} d\tau \hfill\cr\hfill \lesssim \alpha\|V_\varepsilon\|_{L_t^\infty(\dot B_{2,1}^{\frac{1}{2}}\cap \dot B_{2,1}^\frac{3}{2})}+\alpha \int_0^t \|V_\varepsilon\|_{\dot B_{2,1}^{\frac{5}{2}}\cap \dot B_{2,1}^{\frac{7}{2}}} d\tau .} $$

We obtain : $$\displaylines{\|V_\varepsilon(t)\|_{L_t^\infty(\dot B_{2,1}^{\frac{1}{2}}\cap \dot B_{2,1}^\frac{3}{2})}+\int_0^t \|V_\varepsilon\|_{\dot B_{2,1}^{\frac{5}{2}}\cap \dot B_{2,1}^\frac{7}{2}} d\tau \hfill\cr  \lesssim \|\overline{v}_0-v_0\|_{\dot B_{2,1}^{\frac{1}{2}}\cap \dot B_{2,1}^{\frac{3}{2}}}+\alpha \|V_\varepsilon\|_{L_t^\infty(\dot B_{2,1}^{\frac{1}{2}}\cap \dot B_{2,1}^{\frac{3}{2}})} +\alpha \int_0^t \|V_\varepsilon\|_{\dot B_{2,1}^{\frac{5}{2}}\cap \dot B_{2,1}^{\frac{7}{2}}}d\tau.}$$ 
As the last two terms of the right-hand side are negligible compared to those of the left-hand side for $\alpha$ small enough, we obtain for all $t\in\R_+$ : $$\|V_\varepsilon\|_{L_t^\infty(\dot B_{2,1}^{\frac{1}{2}}\cap \dot B_{2,1}^\frac{3}{2})}+\int_0^t \|V_\varepsilon\|_{\dot B_{2,1}^{\frac{5}{2}}\cap \dot B_{2,1}^\frac{7}{2}} d\tau  \lesssim \varepsilon,$$

 whence the result.

\appendix
\section{}
We recall here classical lemmas on differential equations and two Poincaré inequalities in the vertical direction.
\begin{lemma}\label{lemme edo}
Let $\displaystyle X:[0,T]\rightarrow \R_+$ a continuous function such that $X^2$ is derivable. Suppose there is a constant $c\geq 0$ and a measurable function $A:[0,T]\rightarrow \R_+$ such that $$\frac{1}{2}\frac{d}{dt}X^2+c X^2\leq A X \quad pp\ \text{on} \ [0,T].$$
Then, for all $t\in [0,T]$, we have: $$X(t)+c\int_0^t X(\tau)\,d\tau\leq X_0 +\int_0^t A(\tau) \,d\tau.$$
\end{lemma}

The following result is classic: see for example \cite{RD}.
\begin{lemma}\label{lemme edo2}
    Let $T>0$. Let $\mathcal{L}:[0,T]\to \R$ and $H:[0,T]\to \R$ two positive continuous functions on $[0,T]$ such that $\mathcal{L}(0)< \alpha$ with $\displaystyle\alpha\in ]0,\frac{c}{2C}[$ and $$\mathcal{L}(t)+c\int_0^t \mathcal{H}(\tau)d\tau \leq \mathcal{L}_0+C\int_0^t \mathcal{L}(\tau)\mathcal{H}(\tau)d\tau.$$
    Then, for all $t\in[0,T]$, we have : $$\mathcal{L}(t)+\frac{c}{2}\int_0^t \mathcal{H}(\tau)d\tau \leq \mathcal{L}(0).$$
\end{lemma}

We recall two of Poincaré's inequalities: 
\begin{lemma}\label{inégalité poincaré-wirtinger}~\\
    Let $f\in \mathcal{C}_0^\infty(\Omega)$, we have : $$\left|f(\cdot,z)-\frac{1}{2}\int_{-1}^1 f(\cdot, z) \ dz\right|\leq 2 |\partial_z f(\cdot, z)|.$$
    Moreover, if $f$ is odd with respect to the vertical variable, then we have $$|f(\cdot,z)|\leq 2 |\partial_z f(\cdot, z)|.$$ 
\end{lemma}

\begin{lemma}\label{inégalité poincaré}~\\
    Let $f\in \mathcal{C}_0^\infty(\Omega)$, we have: $$\|f\|_{L^2(\Omega)}\leq 2 \|\partial_z f\|_{L^2(\Omega)}.$$
\end{lemma}

\section{}
Here we recall the construction of Besov spaces and some of their properties.

In this article, we used a classical decomposition in Fourier space, called Littlewood Paley's homogeneous dyadic decomposition $(\dot\Delta_j)_{j\in\Z}$ defined by $\dot\Delta_j\mathrel{\mathop:}= \varphi(2^{-j}D).$ Here, we consider $\varphi$ and $\chi$ two regular functions representing a partition of the unit in $\R$ verifying the proposition 2.10 of \cite{RD} such that $\mathrm{supp} \ \chi\subset B(0,\frac{4}{3})$, $\mathrm{supp} \ \varphi \subset\mathcal{C}\mathrel{\mathop:}=\{\xi\in\R^d,\:  3/4\leq|\xi|\leq 8/3\}$ and satisfying $$\sum_{j\in\Z}\varphi(2^{-j}\xi)=1,\qquad \xi\not=0.$$ 

By construction, $\dot\Delta_j$ is a localization operator around the frequency of magnitude $2^j$.

For all $j\in\Z$, dyadic homogeneous blocks $\dot \Delta_j$ and the low-frequency truncation operator $\dot S_j$ are defined by 

\begin{eqnarray}\label{déf bloc dyadique classique}\dot \Delta_j u \mathrel{\mathop:}=\mathcal{F}^{-1}(\varphi(2^{-j}\cdot)\mathcal{F}u), \quad \dot S_j u\mathrel{\mathop:}=\mathcal{F}^{-1}(\chi(2^{-j}\cdot)\mathcal{F}u),\end{eqnarray} 
where $\mathcal{F}$ and $\mathcal{F}^{-1}$ denote the Fourier transform and its inverse respectively. From now on, we will use the following shorter notation : $$ u_j:=\dot\Delta_j u.$$

Let $\mathcal{S}_h'$ the set of tempered distribution $u$ on $\R^d$ such that $$\displaystyle \underset{j\to -\infty}{\lim}\|\dot S_j u\|_{L^\infty}=0.$$ we then have : $$u=\sum_{j\in \Z} u_j \ \in \mathcal{S}', \quad \dot S_j u=\sum_{j'\leq j-1} u_{j'}, \ \forall u\in \mathcal{S}_h'.$$

With the help of these dyadic blocks, the homogeneous Besov spaces $\dot B_{2,1}^s$ for all $s\in\R$ are defined by : $$\dot B_{2,1}^s\mathrel{\mathop:}=\left\{u\in \mathcal{S}_h' \middle| \|u\|_{\dot B_{2,1}^s}\mathrel{\mathop:}=\|\{2^{js}\|u_j\|_{L^2}\}_{j\in\Z}\|_{l^1}<\infty\right\}.$$

A generalization of these properties on the torus has been realized in \cite{Ref Besov sur le tore 1}, \cite{RD mini-cours} and \cite{Ref Besov sur le tore 2} and we admit their adaptation on $\Omega_2$. In this context, we define \eqref{déf bloc dyadique classique} by : $$\dot\Delta_j u(x,y,z)=\sum_{n\in\Z}\mathcal{F}_{H}^{-1}(\varphi(2^{-j}\cdot,2^{-j}n)\mathcal{F}_H u)(x,y)\times \widehat{u}_n e^{i\pi nz}$$ where  $\displaystyle\widehat{u}_n=\frac{1}{2} \int_{-1}^1 e^{-i \pi nz}u(x,y,z)dz$ and $\mathcal{F}_H$ is the Fourier transform in the horizontal component.

The following lemma is a classical result of product laws on Besov spaces, see for example \cite{BCD}.
\begin{lemma} \label{Produit espace de Besov}
For $d\geq 2$, the numerical product extends into a continuous application from $\dot B_{2,1}^{\frac{d}{2}-1}\times \dot B_{2,1}^{\frac{d}{2}}$ to $\dot B_{2,1}^{\frac{d}{2}-1}$. 

$ \dot B_{2,1}^{\frac{d}{2}}$ is a multiplicative algebra for $d\geq 1$.

For $d\geq 1$, we have for $(u,v)\in \dot B_{2,1}^{\frac{d}{2}}\cap  \dot B_{2,1}^{\frac{d}{2}+1}$ that $uv\in  \dot B_{2,1}^{\frac{d}{2}+1}$ and the following inequality : $$\|uv\|_{ \dot B_{2,1}^{\frac{d}{2}+1}}\lesssim \|u\|_{ \dot B_{2,1}^{\frac{d}{2}}}\|v\|_{ \dot B_{2,1}^{\frac{d}{2}+1}}+\|u\|_{ \dot B_{2,1}^{\frac{d}{2}+1}}\|v\|_{ \dot B_{2,1}^{\frac{d}{2}}}.$$
\end{lemma}

\end{document}